\documentclass[12pt,reqno]{amsart}
\usepackage{graphicx}
\usepackage{psfrag}
\usepackage{pxfonts}
\usepackage{mathrsfs}
\usepackage{color}
\usepackage{amsmath,amssymb}

\numberwithin{equation}{section}

\theoremstyle{plain}
\newtheorem{maintheorem}{Theorem}

\newtheorem{theorem}{Theorem}[section]

\newtheorem{corollary}[theorem]{Corollary}
\newtheorem{proposition}[theorem]{Proposition}
\newtheorem{lemma}[theorem]{Lemma}
\newtheorem{definition}[theorem]{Definition}

\theoremstyle{remark}
\newtheorem{remark}[theorem]{Remark}

\begin{document}

\thanks{}

\author{F. Micena}
\address{
  IMC-UNIFEI Itajub\'{a}-MG, Brazil.}
\email{fpmicena82@unifei.edu.br}

\author{A. Tahzibi}
\address{ICMC-USP S\~{a}o Carlos -SP, Brazil.}
\email{tahzibi@icmc.usp.br}


\renewcommand{\subjclassname}{\textup{2000} Mathematics Subject Classification}

\date{\today}

\setcounter{tocdepth}{2}

\title{A Note on Rigidity of Anosov diffeomorphisms of the Three Torus}
\maketitle
\begin{abstract}
We consider Anosov diffeomorphisms on $\mathbb{T}^3$ such that the tangent bundle splits into three subbundles $E^s_f \oplus E^{wu}_f \oplus E^{su}_f.$ We show that if $f$ is  $C^r, r \geq 2,$  volume preserving, then $f$ is $C^1$ conjugated with its linear part $A$ if and only if the center foliation $\mathcal{F}^{wu}_f$ is absolutely continuous and the equality  $\lambda^{wu}_f(x) = \lambda^{wu}_A,$ between center Lyapunov exponents of $f$ and $A,$ holds for $m$ a.e. $x \in \mathbb{T}^3.$ We also conclude rigidity of derived from Anosov diffeomorphism, assuming an strong absolute continuity property (Uniform bounded density property) of strong stable and strong unstable foliations.

\end{abstract}

\section{Introduction}
 We consider $f: \mathbb{T}^3 \rightarrow \mathbb{T}^3$  an Anosov diffeomorfism. Denote by $A$ the linearization of $f,$ the map induced on $\mathbb{T}^3$ by the matrix with integer coefficients given by the action of $f$ on $\Pi_1(\mathbb{T}^3).$ It is known by \cite{FRANKS} that $A$ is an Anosov automorphism, and $f$ and $A$ are conjugated by a homeomorphism $h$ such that
 $$h \circ f = A \circ h.$$

 We also consider  $f$ as a partially hyperbolic diffeomorphism with splitting $T\mathbb{T}^3 = E^s_f \oplus E^{wu}_f \oplus E^{su}_f,$ where $E^s$ is the uniform contracting direction, $E^{wu}$ is the weak unstable direction and $E^{su} $ is the strong unstable direction. In partially hyperbolic setting $E^{wu}$ can be seen as the center bundle.

 From now on we consider that $f $ is $C^r, r \geq 1,$ and $m-$ preserving, where $m$ is the Lebesgue measure on $\mathbb{T}^3.$

 The main question that we  treat here is about ``minimal" conditions imposed on $f$  in order to have $h$ a $C^1-$conjugacy.   Gogolev and Guysinsky \cite{GoGu} proved that in the absence of periodic data obstruction,  there is  $C^1-$conjugacy and even $C^{r}$ if $f$ is $C^r$ (Gogolev \cite{Go2}). Previously with analogous hypothesis De la Llave in \cite{LLAVE}, proved the $C^1-$conjugacy for Anosov diffeomorphisms on $\mathbb{T}^2,$ and showed that the coincidence of periodic data is not sufficient to get $C^1$ conjugacy for Anosov diffeomorphisms on $\mathbb{T}^4.$
During preparation of this work Saghin and Yang \cite{SY} announced also some rigidity results and in particular they prove that if the Lyapunov exponents (stable, weak unstable and strong unstable) of Lebesgue almost every point coincide with the exponents of linearization, then there is a smooth conjugacy. Our result is  similar in spirit to theirs, although it is announced under different hypothesis and implies their result as a corollary (it does not mean necessarily being stronger).
We also refer to M. Poletti \cite{poletti} result which is in the same spirit.

In this note, we address some other point of views which are not mentioned in the previous works and may be interesting to be investigated more.

 By Brin, Burago and Ivanov, in \cite{BBI},  (Potrie-Hammerlindl for non absolute case) the diffomorphism   $f$ is dynamically coherent, meaning the center bundle $E^{wu}_f$ is uniquely integrable to a one dimensional foliation $\mathcal{F}^{wu}_f.$ {
 $A$ has a partially hyperbolic decomposition $T\mathbb{T}^3 = E^s_A \oplus E^{wu}_A \oplus E^{su}_A, $ moreover, by \cite{GoGu} Lemma 2, we have  $$h(\mathcal{F}^{wu}_f(x)) = \mathcal{F}^{wu}_A(h(x)) =  \{h(x)\} + E^{wu}_A.$$

If $h$ is a $C^1-$diffeomorphism, the foliation $\mathcal{F}^{wu}_f$ is absolutely continuous, meaning that the conditional measures of $m$ along the center leaves  $\mathcal{F}^{wu}_f(x)$ are absolutely continuos with respect to $Vol_{\mathcal{F}^{wu}_f(x)},$ the volume  restricted to the leaf $\mathcal{F}^{wu}_f(x).$

Here we prove that:
\begin{maintheorem} Let $f$  be  a  $C^r, r \geq 2,$  volume preserving Anosov diffeomorphism such that $T\mathbb{T}^3 = E^s_f \oplus E^{wu}_f \oplus E^{su}_f.$ Then $f$ is $C^1$ conjugated with its linear part $A$ if and only if the center foliation $\mathcal{F}^{wu}_f$ is absolutely continuous and the equality $\lambda^{wu}_f(x) = \lambda^{wu}_A,$ between center Lyapunov exponents of $f$ and $A,$ holds for $m$ a.e. $x \in \mathbb{T}^3.$
\end{maintheorem}

In the above theorem we have assumed absolute continuity of center foliation and some periodic data to obtain rigidity. We emphasize that only the absolute continuity  of $\mathcal{F}^{wu}_f(x)$ is not sufficient to ensure that $h$ is $C^1.$ In fact, in \cite{VARAO} there are examples of $C^{\infty}-$Anosov diffeomorphism, $f,$ such that $\mathcal{F}^{wu}_f$ is absolutely continuous but $f$ is not $C^1$ conjugated with its linearization $A.$

We have defined the notion of foliations with Uniform Bounded Density in \cite{MT} as follows:

\begin{definition}[\cite{MT}] We say that a foliation $\mathcal{F}$ has Uniform Bounded Density property (or UBD for short) if there is a uniform constant $K > 1,$ such that for any $\mathcal{F-}$foliated box $B$ (independent of the size of the box)
$$ K^{-1}  \leq \frac{dm_x^B}{d \widehat{\lambda_x}} \leq K$$
where $m_x^B$ is the conditional measure of volume on the leaves of $\mathcal{F}$ on $B$ and  $\widehat{\lambda_x}$  means the normalized Lebesgue measure  on the connected component  of $\mathcal{F}(x) \cap B$ wich contains $x.$
\end{definition}

In what follows we consider the derived from Anosov setting (partially hyperbolic with Anosov linearization) on $\mathbb{T}^3$. We  concentrate our hypotheses on stable and unstable foliations. In fact, it is well-known that both stable and unstable foliations of any $C^2$ partially hyperbolic diffeomorphism are absolutely continuous.

\begin{maintheorem}
Let $f: \mathbb{T}^3 \rightarrow \mathbb{T}^3 $ be a $C^r, r\geq 2, $  derived from Anosov (DA) diffeomorphism, volume preserving, with linearization $A,$ such that $T\mathbb{T}^3 = E^s_A \oplus E^{wu}_A \oplus E^{su}_A. $ Suppose that $\mathcal{F}^{su}_f$ and $\mathcal{F}^s_f$ are minimal foliations with the UBD (Uniformly Bounded Density) property. Then $f$ is Anosov and it is $C^1-$conjugated with $A.$
\end{maintheorem}

In the Anosov setting it is easy to prove the following corollary of Theorem A.
\begin{corollary} \label{anosovUBD}
 If  $f$ is Anosov and $\mathcal{F}^{wu}$ has UBD property, then $f$ and $A$ (its linearization) are $C^1-$conjugated.
\end{corollary}

We would like to say that R.Var\~{a}o had proved one of our previous conjectures  which relates rigidity with UBD property of center foliation:

\begin{theorem} \cite{VARAO2}
Let $f : \mathbb{T}^3 \rightarrow \mathbb{T}^3$ be a conservative partially hyperbolic dif- feomorphism homotopic to a linear Anosov. The center foliation has the Uniform Bounded Density property if and only if f is $C^{\infty}$ conjugate to its
linearization.
\end{theorem}



\section{Preliminaries}
Let $M$ be a $C^{\infty}$ Riemannian closed (compact, connected and boundaryless)  manifold.
A $C^1-$diffeomorphism $f: M \rightarrow M$ is called a partially hyperbolic diffeomorphism if the tangent bundle $TM$ admits a $Df$ invariant tangent decomposition $TM =  E^s \oplus E^c \oplus E^u$ such that all unitary vectors $v^{\sigma} \in E^{\sigma }_x, \sigma \in \{s,c,u\}$ for every $x \in M$ satisfy:

$$ ||D_x f v^s || < ||D_x f v^c || < ||D_x f v^u ||,$$
moreover,

$$||D_x f v^s || < 1 \;\mbox{and}\; ||D_x f v^u || > 1. $$



When $TM =  E^s  \oplus E^u,$ as above, then $f$ is called an Anosov diffeomorphism. However, in this paper, we are considering Anosov diffeomorphisms such that $E^u$ decomposes into two invariant subbundle  (weak and strong unstable) $E^u = E^{wu} \oplus E^{su}$ and we call $E^{wu}$ as central bundle.


\begin{definition} We say that a partially hyperbolic diffeomorphism $f: \mathbb{T}^3 \rightarrow \mathbb{T}^3$ is a derived from Anosov (DA) diffeomorphism if its linearization $A:\mathbb{T}^d \rightarrow \mathbb{T}^d$ is an Anosov diffeomorphism with three invariant subbundle. $A$ is obtained by the extension of induced action of $f$ on $\Pi_1(\mathbb{T}^3).$
\end{definition}

For our purposes, we need to exploit  conditional measures of the Lebesgue measure $m$ on center leaves of an Anosov diffeomorphis as stated in Theorem A. Let us  introduce some  basic concepts about disitegration of a measure.

Let $(M, \mu, \mathcal{B})$
be a probability space, where $M$ is a compact metric space, $\mu$ a probability
measure and B the Borelian $\sigma-$algebra. Given a partition $\mathcal{P}$ of  $M$ by measurable
sets, we associate the probability space $(P, \tilde{\mu}, \tilde{\mathcal{B}})$ by the following
way. Let $\pi: M \rightarrow \mathcal{P}$ be the canonical projection, that is, $\pi$ associates to a
point $x$ of $M$ the partition element of $\mathcal{P}$ that contains it. Then we define
$\tilde{\mu} := \pi_{\ast}\mu$ and $\tilde{\mathcal{B}} := \pi_{\ast}(\mathcal{B}).$

\begin{definition}  Given a partition $\mathcal{P}.$ A family $\{\mu_P \}_{P \in \mathcal{P}}$ is a system of
conditional measures for $\mu$ (with respect to $\mathcal{P}$) if:
\begin{enumerate}
\item given $\varphi \in
C^0(M),$  then $P \mapsto \int \varphi d\mu_P$ is measurable;
\item  $\mu_P(P) = 1,$ $\tilde{\mu}$-a.e.;
\item  if $\varphi \in C^0(M)$, then $\int_M \varphi d \mu =
\int_{\mathcal{P}} \varphi d\mu_p d\tilde{\mathcal{\mu}}.$
\end{enumerate}
\end{definition}
When it is clear which partition we are referring to, we say that the family
$\{\mu_P \}$ disintegrates the measure $\mu.$

\begin{proposition} [\cite{Ro}] Given a partition $\mathcal{P}$, if $\{\mu_P \}$ and $\{\nu_P \}$ are conditional
measures that disintegrate $\mu$ on $\mathcal{P},$ then $\mu_P = \nu_P,$ $\mathcal{\mu}$-a.e.
\end{proposition}
\begin{corollary}
 If $T : M \rightarrow M$ preserves a probability $\mu$ and the partition
$\mathcal{P},$ then $T_{\ast}\mu_P = \mu_{T(P)},$ for
$\widetilde{\mu}$-a.e.
\end{corollary}
\begin{proof}
It follows from the fact that $\{T_{\ast}\mu_P \}_{P \in \mathcal{P}}$ is also a disintegration of
$\mu.$
\end{proof}
\begin{definition} We say that a partition $\mathcal{P}$ is measurable (or countably
generated) with respect to $\mu$ if there exist a measurable family $\{A_i\}_{i\in \mathbb{N}}  $ and
a measurable set $F$ of full measure such that if $B \in \mathcal{P},$ then there exists a
sequence $\{B_i\},$ where $B_i \in \{A_i
, A_i^c \}$ such that $B \cap F =
\cup_{i=1}^{\infty}B_i \cap F.$
\end{definition}

\begin{theorem}[Rokhlin’s disintegration \cite{Ro}] Let $\mathcal{P}$ be a measurable partition
of a compact metric space $M$ and $\mu$ a Borelian probability. Then there
exists a disintegration by conditional measures for $\mu.$
\end{theorem}

Now we discuss some definitions and results concerning entropies of expanding foliations  \cite{LEDRAPPIER} and \cite{CHINESES} and the results by \cite{GoGu} and \cite{Go} about rigidity and absolutely continuous of the $\mathcal{F}^{wu}_f.$

 At 1980's Ledrappier developed a theory about entropy along  expanding foliations. Consider $f: M \rightarrow M$ a $C^1$ diffeomorphism defined on compact, connected and boundaryless $C^{\infty}-$ manifold $M.$ Suppose that $\mathcal{W}$ is an invariant expanding foliation, that is, $f(\mathcal{W}) = \mathcal{W}$ and $||Df| T\mathcal{W}|| > \lambda > 1.$

 \begin{definition} Given a manifold $M,$ a Borel measure $\mu$ and a foliation $W$ on it. We
say that the measurable partition $\xi$ is subordinate to the foliation $W$ whenever for
$\mu$ a.e. we have:
\begin{enumerate}
\item $\xi(x) \subset W_x,$ where $W_x$ is the leaf of the foliation $W,$ through $x.$
\item  $\xi(x)$ contains an open subset of a neighborhood of $x$ in $W_x$
 in the submanifold topology.
\end{enumerate}
 If $W$ is an $f-$ invariant foliation, we say that $\xi$ is increasing and subordinate to $W$ if it is
subordinate to $W$ and:
\begin{itemize}
\item $f^{-1}(\xi)$ refines $\xi.$
\item $\bigvee_{n = 0}^{+\infty} f^{-n}(\xi)$ is the partition into points.
\item The biggest $\sigma-$algebra contained in $\cap_{n =0 }^{+\infty}f^{n}(\xi)$ is the $\sigma-$algebra whose elements are
unions of $W$ manifolds.
 \end{itemize}
\end{definition}

\begin{definition} We say that $\mu_x^{\xi}$
 is the measure conditioned by the partition $\xi$ if for
every measurable $A \subset  M,$ the map $x \mapsto
\mu_x^{\xi}(A)$ is measurable with respect to the $\sigma-$algebra
generated by $\xi$ and $\mu(A) = \int_A \mu_x^{\xi}(A)d\mu(x).$
\end{definition}

\begin{lemma}[\cite{LEDRAPPIER}]\label{increasing} If $\mathcal{W}$ is an $f-$invariant expanding foliation, there exists $\xi$ an increasing and subordinated partition of $\mathcal{W}.$
\end{lemma}

Let $f$ and $\xi$ be as in the above lemma, the number

$$H_{\mu}(f^{-1}\xi | \xi) = - \int \log(\mu_x^{\xi}(P_{f^{-1}\xi}(x))) d\mu(x),$$
does not depend on the increasing and subordinated partition of $\mathcal{W}.$ Here $P_{f^{-1}\xi(x)}$ is the element of $f^{-1}\xi$ that contains $x.$ The metric entropy along the  foliation $\mathcal{W}$ is defined by below
  $$h_{\mu}(f, \mathcal{W}) = H_{\mu}(f^{-1}\xi | \xi) = - \int \log(\mu_x^{\xi}(P_{f^{-1}\xi}(x))) d\mu(x). $$

Let $\mathcal{W}$ be an $f-$invariant expanding foliation. Ledrappier in \cite{LEDRAPPIER} proved the following theorem.

 \begin{theorem}[\cite{LEDRAPPIER}, Theorem 4.8]\label{ledra} Let $\mu$ be a measure without zero Lyapunov exponents. The following statements are equivalent:
 \begin{enumerate}
 \item $h_{\mu}(f, \mathcal{W})  = \int \sum_{i = 1} \lambda_i^{+}(x)dim(E^i(x)) d\mu.$
 \item the disintegration of $\mu$ with respect to  $\mathcal{W}$ is absolutely continuous.
 \end{enumerate}
 \end{theorem}

In the context of the previous Theorem, we have $\mu_x^{\xi} = \rho dVol_{P_{\xi}(x)}$ and $$\frac{\rho(x)}{\rho(y)} = \Delta(x,y) =
\prod_{i=0}^{+\infty}\frac{Jf^{\mathcal{W}}(f^{-i})(x)}{Jf^{\mathcal{W}}(f^{-i})(y)},$$
where $Jf^{\mathcal{W}}(x)$ is the Jacobian of $Df(x)| T\mathcal{W}(x).$
\begin{lemma}\label{continous}
When $\mathcal{W}$ is one dimensional the map $x \mapsto \rho dVol_{P_{\xi}(x)} $ is continuous.
\end{lemma} See \cite{poletti}, remark 2.5.


For our purposes we need to relate conjugacy and absolute continuity with the Lyapunov exponents of the diffeomorphisms.  Let us recall important results from \cite{GoGu} and \cite{Go}. The first Theorem is about $C^1$ conjugacy between Anosov maps of  $\mathbb{T}^3$ and the other one is a characterization of absolute continuity of $\mathcal{F}^{wu}_f.$

\begin{theorem}[\cite{GoGu}] \label{gogolev1} Let $f$  be  a  $C^r, r \geq 2,$ an Anosov diffeomorphism such that $T\mathbb{T}^3 = E^s_f \oplus E^{wu}_f \oplus E^{su}_f,$ and $A$ its linearization. Then $f$ and $A$  are $C^1-$ conjugated if and only if $\lambda^{\sigma}_f(p) = \lambda^{\sigma}_A, \sigma \in \{s, wu, su\},$ for any $p \in Per(f).$
 \end{theorem}

\begin{theorem}[\cite{Go}] \label{gogolev2} Suppose $f$ and $A$ as in Theorem A. The center foliation $\mathcal{F}^{wu}_f$ is absolutely continuous if and only if there is a real  number $\lambda^{su}_f,$ such that $\lambda^{su}_f(p) =  \lambda^{su}_f,$ for any  $p \in Per(f).$
 \end{theorem}







\section{Proof of Theorem A}

\begin{proof}

For the proof of Theorem A, we need  to recall the following  closing lemma.

\begin{lemma}[Anosov Closing Lemma] \label{closing}
Let $f: M \rightarrow M$ be a $C^{1+\alpha}$ diffeomorphism preserving a hyperbolic Borel probability measure. For all $\delta > 0 $ and $\epsilon > 0$ there exists $\beta= \beta(\delta, \epsilon) >0$ such that if $x, f^{n(x)}(x) \in \Delta_{\delta}$ (Pesin block) for some $n(x) >0$ and $d(x, f^{n(x)}(x)) < \beta$ then there exists a hyperbolic periodic point of period $n(x)$, $z$ with $d(f^k(x), f^k(z)) \leq \epsilon$ for all $0 \leq k \leq n(x)-1.$
 \end{lemma}

 Note that, for the Anosov diffeomorphism $f$ as in Theorem A, we may take (a.e)  $\Lambda_{\delta} = \mathbb{T}^3$ for some $\delta>0.$

By Lemma \ref{increasing}, let $\mathcal{\xi}$ be an increasing partition subordinated to $\mathcal{F}^{wu}_f.$ The conjugacy $h$ sends $\xi$ to a partition $h(\xi) $ increasing and subordinated to $\mathcal{F}^{wu}_A,$ it is because $h$ is homeomorphism and $h(\mathcal{F}^{wu}_f) = \mathcal{F}^{wu}_A,$ see for instance \cite{GoGu}.

Now, since $\mathcal{F}^{wu}_f$ is absolutely continuous, the disintegration of $m$ along $\mathcal{F}^{wu}_f$ has absolutely continuous conditional measures. Then,
$$h_{m}(f, \mathcal{F}^{wu}_f ) = \int \lambda^{wu}_f(x)dm = \lambda^{wu}_A. $$
Consider $\nu$ the induced measure $\nu := h_{\ast}(m).$ Since $h(\mathcal{F}^{wu}_f) = \mathcal{F}^{wu}_A,$ then $h$ conjugates $f$ and $A,$ along the center manifolds, so
$$h_{\nu}(A, \mathcal{F}^{wu}_A ) = h_{m}(f, \mathcal{F}^{wu}_f )=  \lambda^{wu}_A= \int \lambda^{wu}_A(x)d\nu.$$
In fact, every $x\in \mathbb{T}^3$ is  a regular point (Lyapunov exponents exist.) for $A$ and the Lyapunov exponents of $A$ are constant for any $x.$ By Ledrappier formula $\nu$ disintegrates absolutely continuous on each atom of $h(\xi).$  Moreover, since the derivative of $A$ is constant along to $E^{wu}_A,$ by formula in Theorem \ref{ledra}, the densities  of $\nu$ along $\mathcal{F}^{wu}_A$ are constant.  Since $h $ is a homeomorphism, and $supp(m) = \mathbb{T}^3,$ then $supp(\nu) = \mathbb{T}^3.$

\textit{Claim 1: $h$ is $C^{1}$ restricted to the center foliation $\mathcal{F}^{wu}_f.$}

Let $m_x^{\xi}$ be each conditional measure of $m$ along the atom $P_{\xi}(x),$ and $$\rho(x) : = \rho^{\xi}(x)= \frac{d m_x^{\xi}}{d Vol_{P_{\xi}(x)}},$$ the density of $m$  along $P_{\xi}(x),$ where $Vol_{P_{\xi}(x)}$ denotes the normalized Lebesgue measure on $P_{\xi}(x).$

By uniqueness of the Rokhlin disintegration, we have $h_{\ast}(m_x^{\xi}) = \nu^{h(\xi)}_{h(x)}. $ By continuity in  Lemma \ref{continous}, we have
$$\frac{d \nu^{h(\xi)}_x}{d Vol_{h(P_{\xi}(x))}} = Const = 1. $$

 So, if $[x, y]^{wu}$ is a segment in $P_{\xi}(\theta),$ the image $[h(x), h(y)]$ is a segment in $\{h(x)\} + E^{wu}_A. $ Choosing a convenient orientation, we have

$$ h(y) - h(x) = \int_x^y \rho dVol_{[x, y]^{wu}}.$$
It is because, in local charts,  the value $h(y) -  h(x) $  is the length of $[h([x, y]^{wu})]$ as an interval of $\mathbb{R}.$  So $h$ is differentiable on $[x, y]^{wu}$  for  $m$ a.e. $x \in \mathbb{T}^ 3,$ moreover
$$\int_x^y h'(\theta) dVol_{[x,y]^{wu}} =  \int_x^y \rho dVol_{[x, y]^{wu}},$$ for any $y \in {P_{\xi}(x)}.$ The notation $h'$ means the derivative of $h$ in the $wu-$direction.

Each leaf of $\mathcal{F}^{wu}_f$ is a $C^{1+ \alpha}, \alpha > 0,$ submanifold. In local charts $[x, y]^{wu}$  can be seen as an interval.

So, $h'(y) = \rho(y).$ By Lemma \ref{continous},  $h'(y)$ varies continuously for $m-$a.e $y \in \mathbb{T}^ 3.$ Since $\rho$ varies continuously, we extend $h'$ using an Arzela-Ascoli argument type.

We conclude that $h$ is $C^1$ restricted to $\mathcal{F}^{wu}_f.$

\textit{Claim 2: The conjugacy $h$ is $C^{1}.$}

Since  $\mathcal{F}^{wu}_f$ is absolutely continuous,  by  Theorem \ref{gogolev2}, we have  that $\lambda^{wu}_f$ is constant on $Per(f).$ Since $h$ is $C^1$ along $wu-$foliation $\lambda^{wu}(p) = \lambda^{wu}_A := \lambda^{wu}_f, \forall p \in Per(f).$

As $\mathcal{F}^{wu}_f$ is absolutely continuous, so using Theorem \ref{gogolev2}, there is a number $\lambda^{su}_f$ such that $\lambda^{su}(p)=  \lambda^{su}_f,$ for any $p \in Per(f).$ Since the sum of the Lyapunov exponents is zero, then $\lambda^s(p)$ is also  constant $\lambda^s_f,$ for any  $ p \in Per(f).$

Here we emphasize that, the constance of $\lambda^{su}_f$ on $Per(f)$ (strong unstable Lyapunov exponent of periodic points) does not ensure directly that such constant is $ \lambda^{su}_A.$

Now if $\mu$ is an invariant hyperbolic probability measure of $f,$ since $\mu$ a.e. point $x $ is recurrent for $f,$  applying Closing Lemma \ref{closing}, we get  $$\lambda^{\sigma}_f(x) = \lambda^{\sigma}_f, \sigma \in\{s,wu,su\}$$ for $\mu-$a.e. $x \in \mathbb{T}^3.$

By Ruelle's inequality, we have
$$h_{\mu}(f) \leq \lambda^{wu}_f + \lambda^{su}_f.$$

Using the  Pesin's entropy formula we have $$h_m(f) = \lambda^{wu}_f + \lambda^{su}_f.$$

So  $m$ is the maximal entropy measure for $f.$  Using the above expression and the fact $\lambda^{wu}_f = \lambda^{wu}_A,$ we have that

$$\lambda^{su}_f = \lambda^{su}_A,$$
consequently
$$\lambda^{s}_f = \lambda^{s}_A.$$
By Theorem \ref{gogolev1},  Theorem A is proved.
\end{proof}

\begin{corollary}[\cite{SY}, Theorem A]
 Let $f$  be  a  $C^r, r \geq 2,$ a volume preserving  Anosov map such that $T\mathbb{T}^3 = E^s_f \oplus E^{wu}_f \oplus E^{su}_f,$such that $\lambda^{\sigma}_f(x) = \lambda^{\sigma}_A, \sigma \in \{s, wu, su\}$ for $m-$ a.e. $x \in \mathbb{T}^3.$ Then $f$ is $C^1$ conjugated with $A.$
\end{corollary}

\begin{proof}
By Pesin formula $m$ is the maximal entropy measure of $f.$ Then $(h^{-1})_{\ast}(m) = \nu = m.$
Also
$$h_m(f, \mathcal{F}^{wu}_f)= h_{\nu}(A, \mathcal{F}^{wu}_A) = h_m(A,\mathcal{F}^{wu}_A ) = \lambda^{wu}_A = \int \lambda^{wu}_f(x)dm, $$
the Pesin formula holds for the $wu-$entropy of $f,$ then  by Theorem \ref{ledra} the disintegration of $m$ along $\mathcal{F}^{wu}_f$ is absolutely continuous. Thus $\mathcal{F}^{wu}_f$ is absolutely continuous, and $\lambda^{wu}_f(x) = \lambda^{wu}_A,$ for a.e. $x \in \mathbb{T}^3,$ by Theorem A, the conjugacy $h$ is $C^1.$
\end{proof}

As a corollary of the method in the final part of Theorem A, we can alson prove:

\begin{corollary}
  Let $f: \mathbb{T}^3 \rightarrow \mathbb{T}^3 $ be a $C^r, r \geq 2,$ volume preserving Anosov diffeomorphism such that $T\mathbb{T}^3 =  E^{s}_f \oplus E^{wu}_f \oplus E^{su}_f.$ Suppose that there are constants $\Lambda^{\sigma}_f,  \sigma \in \{s, wu, su\},$ such that  for any $p \in Per(f)$ we have $\lambda^{\sigma}_f(p) = \Lambda^{\sigma}_f, \sigma \in \{s, wu, su\}.$ Then $f$ is $C^1$ conjugated with its linearization $A.$
\end{corollary}

\begin{proof}
First we observe that by using Anosov Closing Lemma, as in Theorem A, we conclude that
$$h_{\mu}(f) \leq \Lambda^{wu}_f + \Lambda^{su}_f, $$
where $\mu$ is an arbitrary Borelian invariant probability measure for $f.$
By Pesin's formula
$$h_{m}(f) = \Lambda^{wu}_f + \Lambda^{su}_f, $$
so $m$ is the measure of maximal entropy of $f,$ it implies that
\begin{equation}
\Lambda^{wu}_f + \Lambda^{su}_f =  \lambda^{wu}_A + \lambda^{su}_A \label{inequality}.
\end{equation}

The constance of $\lambda^{su}_f$ along $Per(f),$ implies that $\mathcal{F}^{wu}_f$ is absolutely continuous. Now using the absolute continuity of $\mathcal{F}^{wu}_f$ and $\mathcal{F}^{su}_f,$ by \cite{MT} we have
$$\Lambda^{\sigma}_f \leq \lambda^{\sigma}_A, \sigma \in \{wu, su\}.$$

So by equation $(\ref{inequality})$, none of the above inequalities can be strict, so we have coincidence of periodic data between $f$ and $A.$ By Theorem \ref{gogolev1}, we conclude that $f$ and $A$ are $C^1-$conjugated.

\end{proof}

\section{Desintegration with Uniform Bounded Density }
In this section we connect rigidity in terms of $C^1-$conjugacy with strong regularity conditions on invariant foliations.

First we prove Corollary \ref{anosovUBD} as a consequence of Theorem A.
\begin{proof}
If $\mathcal{F}^{wu}$ has UBD property,   $\mathcal{F}^{wu}$ is absolutely continuous and by \cite{MT}, $\lambda^{wu}_f(x) =  \lambda^{wu}_A,$ for $m$ a.e. $x \in \mathbb{T}^3.$  We are in the hypotheses of Theorem A, and the corollary is proved.
\end{proof}

\begin{remark} If $f$ is derived from Anosov diffeomorphism and $\mathcal{F}^c$ has the UBD property, it is  possible to prove that $f$ is Anosov. See  \cite{VARAO2}.
\end{remark}

To prove Theorem B, we need the following lemma.

\begin{lemma}\label{lemmaUBD} Let $f: M \rightarrow M$ be a diffeomorphism of a (compact, connected, boundaryless, orientable) $C^{\infty}-$manifold. Consider $\mathcal{F}$ an invariant foliation by $f.$ Suppose that $f$ preserves a volume form $m$ defined on $M.$ If $\mathcal{F}$ is a foliation of $M$ with $C^1-$leaves having the UBD-propety then, for $x, y \in \mathcal{F}(x)$ (the $\mathcal{F}-$leaf by $x$) we have $\displaystyle\prod_{j=0}^{+\infty}\displaystyle\frac{|J^{\mathcal{F}}f^ n(x)|}{|J^{\mathcal{F}}f^n (y)|} \in [K^{-4}, K^4],$ where the constant $K$ is the same as in the definiton of UBD-property and $J^{\mathcal{F}}f(x)$ denotes jacobian of $f$ restricted to the tangent space of $\mathcal{F}(x).$
 \end{lemma}

\begin{proof} Consider $B$ a $\mathcal{F}-$foliated box and $L_x$ the connected component of $\mathcal{F}(x)$ contained in $B.$ Denote by $m_B$ the normalized volume restricted to $B$ and by $\lambda_x$ the volume form induced by the Riemannian structure of $M$ on $\mathcal{F}(x).$ Since $f$ preserves the volume form $m,$ then $f_{\ast}m_B = m_{f(B)}.$ Moreover, if $m_x^B$ denotes the condition measures as in the definition, we have $m_{f(x)}^{f(B)} = f_{\ast}m_x^B.$ Let $\rho(x)$ be the density $\frac{d m_x^ B}{d\hat{\lambda}_x},$ we can compute $\rho(f(x))$ using Lebesgue derivation Theorem. Consider $\varepsilon > 0$ and $B_{\varepsilon}(x)$ the open ball in $M,$ centered in $x$ with radius $\varepsilon > 0.$ We have
$$\frac{m_{f(x)}^{f(B)}(f(B_{\varepsilon}(x)))}{\hat{\lambda}_{f(x)}(f(B_{\varepsilon}(x)))} = \frac{m_x^B(B_{\varepsilon}(x))}{\left(\int_{B_{\varepsilon}(x)}|J^{\mathcal{F}}f(t)| d\lambda_x(t) \right) /|L_{f(x)}|},$$
 where $|L_{f(x)}| = \lambda_{f(x)}(L_{f(x)}) .$  So we have
 $$\frac{m_{f(x)}^{f(B)}(f(B_{\varepsilon}(x)))}{\hat{\lambda}_{f(x)}(f(B_{\varepsilon}(x)))} = |L_{f(x)}|\cdot \frac{m_x^B(B_{\varepsilon}(x))}{\int_{B_{\varepsilon}(x)}|J^{\mathcal{F}}f(t)| d\lambda_x(t)}.$$

Dividing numerator and denominator in right hand of the last expression by $\hat{\lambda}_x(B_{\varepsilon}(x)) = \frac{\lambda_x(B_{\varepsilon}(x))}{|L_x|},$ and passing limit when $\varepsilon \rightarrow 0,$ we obtain
\begin{equation} \label{density1}
\rho(f(x)) = \frac{dm_{f(x)}^{f(B)}}{d \hat{\lambda}_{f(x)}}= \frac{|L_{f(x)}|}{|L_x|}\cdot\frac{\rho(x)}{|J^{\mathcal{F}}f(x)|}.
\end{equation}

 By (\ref{density1}), if $x,y \in L_x,$  we have $L_x = L_y$ and $L_{f(x)} = L_{f(y)}$, so $\frac{\rho(f(x))}{\rho(f(y))} = \frac{\rho(x)}{\rho(y)}\frac{|J^{\mathcal{F}}f(y)|}{|J^{\mathcal{F}}f(x)|}.$ Applying the chain rule and (\ref{density1}) repeatedly, we get

 \begin{equation} \label{density2}
\frac{\rho(f^n(x))}{\rho(f^n(y))} = \frac{\rho(x)}{\rho(y)} \cdot\frac{|J^{\mathcal{F}}f^n(y)|}{|J^{\mathcal{F}}f^n(x)|} = \frac{\rho(x)}{\rho(y)}\cdot \displaystyle\prod_{j=0}^{n-1} \frac{|J^{\mathcal{F}}f(f^j(y))|}{|J^{\mathcal{F}}f(f^j(x))|},
\end{equation}
 for any integer $n \geq 1.$

Since $\frac{\rho(f^n(x))}{\rho(f^n(y))}, \frac{\rho(x)}{\rho(y)} \in [K^{-2}, K^2],$ by definition of UBD, then dividing both sides in (\ref{density2})  by  $\frac{\rho(y)}{\rho(x)},$ we obtain $\displaystyle\prod_{j=0}^{n-1} \frac{|J^{\mathcal{F}}f(f^j(x))|}{|J^{\mathcal{F}}f(f^j(x))|} \in [K^{-4}, K^4],$ for any $n \geq 1.$ It concludes the proof.
\end{proof}

An straightforward consequence of the above lemma is that if $\mathcal{F}$ is one dimensional foliation and $x$ is a regular point for $f,$ then the Lyapunov exponent of $f$ in the  $T\mathcal{F}$ is constant on $\mathcal{F}(x).$ In other words,
$$\displaystyle\lim_{n\rightarrow +\infty} \frac{1}{n} \log(|J^{\mathcal{F}}f^n(x)|)  = \displaystyle\lim_{n\rightarrow +\infty} \frac{1}{n} \log(|J^{\mathcal{F}}f^n(y)|) , \; \forall y \in \mathcal{F}(x). $$

\begin{theorem}[Theorem B]
Let $f: \mathbb{T}^3 \rightarrow \mathbb{T}^3 $ be a $C^r, r\geq 2, $ DA diffeomorphism, volume preserving, with linearization $A,$ such that $T\mathbb{T}^3 = E^s_A \oplus E^{wu}_A \oplus E^{su}_A. $ Suppose that $\mathcal{F}^u_f$ and $\mathcal{F}^s_f$ are minimal foliations with the UBD property. Then $f$ is Anosov and it is $C^1-$conjugated with $A.$
\end{theorem}

\begin{proof}
Denote by $\rho^u_B(x)=  \frac{dm_x^B}{d \widehat{\lambda_x}}$ the density of conditional measure $m^x_B$ w.r.t $\widehat{\lambda_x},$ when we consider a $\mathcal{F}^u_f-$foliated box $B.$
If $y \in \mathcal{F}^u_f(x),$ by Lemma \ref{lemmaUBD} we have

%

$$K^{-4} \leq \frac{|Df^ n(x)|E^u_f(x)|}{|Df^ n(y)|E^u_f(y)|}\leq K^4,$$
for any $x,y$ in the same $u-$leaf, and for any $n \geq 1.$

Now fix $x,z \in \mathbb{T}^3,$ such that $x$ is regular (Lyapunov exponents exist) and consider $n \geq 1 $ an integer number. Consider $\mathcal{F}^u_f(x)$ the $u-$  leaf by $x.$ Using that $\mathcal{F}^u_f(x)$ is dense in $\mathbb{T}^3,$
  it is possible, to choose a sequence $\{y_k\}_{k \in \mathbb{N}}$  of points in $\mathcal{F}^u_f(x),$ convering to a point $z,$ such that for every $k \in \mathbb{N},$ we have  $K^{-1} \leq \frac{Df^n(z)|E^u_f(z)}{Df^n(y_k)|E^u_f(y_k)} \leq K.$  So
 $$\frac{|Df^n(z)|E^u_f(z)|}{|Df^n(x)|E^u_f(x)|} = \frac{|Df^n(z)|E^u_f(z)|}{|Df^ n(y_k)|E^u_f(y_k)|} \cdot \frac{|Df^ n(y_k)|E^u_f(y_k)|}{|Df^n(x)|E^u_f(x)|}, $$
 where $K^{-4} \leq \frac{|Df^ n(y_k)|E^u_f(y_k)|}{|Df^ n(x)|E^u_f(x)|}\leq K^4,$ and $K^{-1} \leq \frac{|Df^n(z)|E^u_f(z)|}{|Df^n(y_k)|E^u_f(y_k)|} \leq K,$ thus
 $$K^{-5}\leq \frac{|Df^n(x)|E^u_f(x)|}{|Df^n(z)|E^u_f(z)|}\leq K^5,$$
 for any $n \geq 1.$ In particular  $\lambda^u_f(z) = \lambda^u_f(x),$ for any $z \in \mathbb{T}^3$.

 In the same way  $\lambda^s_f(z) = \lambda^s_f(x),$ for any $z \in \mathbb{T}^3,$ with fixed $x$ a regular point.

Moreover it is possible to choose a constant $L > 0,$ such that
\begin{equation}\label{uniform}
  L^{-1}\leq \frac{|Df^n(x)|E^{\sigma}_f(x)|}{|Df^n(z)|E^{\sigma}_f(z)|}\leq L, \; \sigma \in \{s,u\}
\end{equation}
for any $n \geq 1$ and for any $z \in \mathbb{T}^3,$ with fixed $x$.

By \cite{MT}, since the foliations  $\mathcal{F}^s_f$ and $\mathcal{F}^u_f$ have the UBD property, then $\lambda^{s}_f(x) = \lambda^{s}_A$ and $\lambda^{u}_f(x) = \lambda^{su}_A$   for any $x \in \mathbb{T}^3.$

We claim that $f$ is really an Anosov diffeormorphism.

Fix $0 < \varepsilon < \frac{\lambda^{wu}_A}{4} $ and  choose an $N \geq 1,$ such that $$\left| \frac{1}{n} \log(|J^{s}f^n(x)|) - \lambda^{s}_A \right| <  \varepsilon,$$
 $$\left| \frac{1}{n} \log(|J^{u}f^n(x)|) - \lambda^{su}_A \right| <  \varepsilon,$$
 for any $n \geq N$ and for all $x \in \mathbb{T}^3.$

Since $f$ is volume preserving,
 by the choice of $\varepsilon,$ for any $n \geq N, $ we have

$$ \frac{1}{n} \log(|J^{u}f^n(x)|) \geq \lambda^{wu}_A - 2\varepsilon > \frac{\lambda^{wu}_A}{2}, $$
so for any $x \in \mathbb{T}^3$ we have
$$ |Df^n(x)| E^c_f(x)| \geq e^{n\frac{\lambda^c_A}{2}}, \; \forall \, n \geq N.$$
which yields that $f$ is uniformly expanding on $E^c_f,$ consequently $f$ is Anosov.
As the function $x \mapsto\lambda^{u}_f(x)$ is constant, we conclude that  $\mathcal{F}^c_f := \mathcal{F}^{wu}_f$ is absolutely continuous, by Theorem \ref{gogolev2}. We also have $\lambda^{wu}_f(x) : = \lambda^{wu}_A, $ for any $x \in \mathbb{T}^3.$
Putting these facts together Theorem A, we conclude the proof.
\end{proof}

Using Lemma \ref{lemmaUBD} and the ideas in the proof of Theorem B, we can prove.

\begin{proposition} Let $f: \mathbb{T}^3 \rightarrow \mathbb{T}^3$ be a volume preserving $C^r, r \geq 2,$ derived from Anosov diffeomorphism. Suppose that $\mathcal{F}^c_f$ has UBD property and $\lambda^c_f(x) > 0,$ for $m$ a.e. point $x\in \mathbb{T}^3,$ then $f$ is an Anosov diffeomorphism and $\mathcal{F}^c_f$ is the weak unstable foliation.
\end{proposition}

\begin{proof}
Let $x$  be a point such that  $\lambda^c_f(x) > 0. $  Since $\mathcal{F}^c_f$ has UBD property, then for every $y \in \mathcal{F}^c_f(x) $ we conclude $\lambda^c_f(y) = \lambda^c_f(x).$ By \cite{H}, the foliation $\mathcal{F}^c_f$ is minimal, then using its minimality, as in the proof of Theorem B, we have $\lambda^c_f(z) = \lambda^c_f(x) := \lambda^c_f,$ for any $z \in \mathbb{T}^3.$  Again, as in the proof of Theorem B, it is possible to choose $N > 0,$ such that $$ | Df^n(z)| E^c_f(z)| \geq \frac{1}{K^5} e^{n\frac{\lambda^c_f}{2}},$$
for any $n \geq N$ and for every $z \in \mathbb{T}^3.$ The constant $K$ is as in the same of definition of UBD.  So $Df|E^c_f$ is uniform expanding and finally $f$ is Anosov.
\end{proof}

From Corollary \ref{anosovUBD} and the above proposition we can conclude, by another approach, the result due to R.Var\~{a}o, \cite{VARAO2}.

\begin{theorem} Let $f: \mathbb{T}^3 \rightarrow \mathbb{T}^3 $ be a $C^r, r \geq 2,$ conservative partially hyperbolic diffeomorphism
homotopic to a linear Anosov. The center foliation has the
Uniform Bounded Density property if and only if f is $C^1-$conjugate to its linearization $A.$
\end{theorem}

In the above theorem, if $f$ is $C^{\infty}$ then, the conjugacy is also $C^{\infty}.$ See \cite{VARAO2} for the argument.

\end{document}